\documentclass[a4paper,11pt,parskip=half*,numbers=noendperiod]{article}
\usepackage[margin=3.5cm]{geometry}
\usepackage[utf8]{inputenc}
\usepackage[T1]{fontenc}
\usepackage[english]{babel}
\usepackage{lmodern}
\usepackage{xparse}

\usepackage{lipsum}
\usepackage[shortlabels]{enumitem}
\usepackage{verbatim}
\usepackage[normalem]{ulem}
\usepackage[dvipsnames]{xcolor}
\usepackage{csquotes}
\usepackage{pdfpages}
\usepackage{easy-todo}
\usepackage{setspace}

\usepackage{graphicx}
\usepackage{caption}
\usepackage{subcaption}
\usepackage{tabularx}
\usepackage{multicol}
\usepackage{amsmath}
\usepackage{mathtools}
\usepackage{stmaryrd}
\usepackage{amsthm, thmtools}
\usepackage{framed}
\usepackage{nameref}
\usepackage[colorlinks,menucolor=blue,linkcolor=blue, citecolor=blue, urlcolor=blue]{hyperref} 
\usepackage[capitalise]{cleveref}% if problems, load cleveref last
\crefname{subsection}{Subsection}{Subsections}
\Crefname{subsection}{Subsection}{Subsections}
\usepackage{amssymb}
\usepackage[mathscr]{euscript}
\usepackage{esint}
\usepackage{esvect}
\usepackage{relsize}

\usepackage{array,tabularx,booktabs}

\usepackage{tikz}
\usepackage{tikz-qtree}
\usepackage{tikz-cd}
\usetikzlibrary{calc}

\usetikzlibrary{fit}
\usepgfmodule{nonlineartransformations}
\usetikzlibrary{curvilinear}
\usepackage[framemethod=tikz]{mdframed}
\usetikzlibrary{cd}
\usetikzlibrary{matrix}
\usetikzlibrary{backgrounds}

\newtheorem{theorem*}{Theorem}
\newtheorem{theorem}{Theorem}[section]

\newtheorem{definition}[theorem]{Definition}
\theoremstyle{definition}

\newtheorem{remark}[theorem]{Remark}

\newcommand{\I}{\text{$\mathrm{I}$}}
\newcommand{\II}{\text{$\mathrm{II}$}}

\newcommand{\N}{\text{$\mathbb{N}$}}

\newcommand{\rca}{\mathsf{RCA}_0}
\newcommand{\aca}{\mathsf{ACA}_0}

\newcommand{\Det}{\mathsf{Det}}

\newcommand{\SI}{\mathrm{S}_{\I}}
\newcommand{\SII}{\mathrm{S}_{\II}}

\usepackage[abbrev,backrefs]{amsrefs}

\newcommand\blfootnote[1]{%
  \begingroup
  \renewcommand\thefootnote{}\footnote{#1}%
  \addtocounter{footnote}{-1}%
  \endgroup
}
\usepackage{soul}
\usepackage{authblk}
\begin{document}

\title{Binary Choice Games and Arithmetical Comprehension}
\author[$^1$]{J. P. Aguilera}
\author[$^2$]{T. Kouptchinsky}
\affil[$^{1,2}$]{Institute of Discrete Mathematics and Geometry, Vienna University of Technology. Wiedner Hauptstrasse 8-10, 1040 Vienna, Austria}	
\affil[$^{2}$]{Corresponding author: \texttt{thibaut.kouptchinsky@tuwien.ac.at}}
\maketitle
 \begin{abstract}
We prove that Arithmetical Comprehension is equivalent to the determinacy of all clopen integer games in which each player has at most two moves per turn.
 \end{abstract}
\setcounter{tocdepth}{2}
%\tableofcontents
\numberwithin{equation}{section}

\section{Introduction}
The purpose of this article is to prove the following theorem:
\begin{theorem}\label{TheoremMain}
  Fix $k \geq 1$. The following are equivalent over $\rca$:
  \begin{enumerate}
    \item Determinacy for binary choice wellfounded game trees.
    \item $\Delta^0_1$-$\Det$ for binary choice game trees.
    \item $(\Sigma^0_1)_k$-$\Det$ for binary choice game trees.
    \item $\aca$.
  \end{enumerate}
\end{theorem}
Here $(\Sigma^0_1)_k$ denotes the $k$th level of the difference hierarchy over $\Sigma^0_1$, i.e.,  the class of sets obtained from $k$ nested differences of open sets.

In the games we consider, we suppose that players alternate turns playing numbers $x_i \in\mathbb{N}$. After infinitely many moves have been played, the winner is determined based on whether $x = (x_0, x_1, x_2, \hdots)$ belongs to some clopen subset of $\mathbb{N}^\mathbb{N}$. The additional condition is that in each turn, each player has at most two legal moves they can make, with the provision that an illegal move results in losing the game. The games are defined more carefully in \S\ref{SectionDefinitions} using the mechanism of \textit{game trees}: the game $G(T,A)$ is the game in which player I's winning set is $A$, but both players are required to play within the tree $T$.

Theorem \ref{TheoremMain} is to be contrasted with previous results concerning similar games. Nemoto, MedSalem, and Tanaka \cite{NeMSTa07} have shown that the determinacy of all $\Delta^0_1$ games with moves in $\{0,1\}$ is equivalent to $\mathsf{WKL}_0$. Simpson \cite{Simpson} has shown that the determinacy of all games of length $k$ with integer moves is equivalent to $\aca$ (for a fixed $k \geq 3$). By a well known theorem of Steel \cite{St77}, $\Delta^0_1$-Determinacy of games on $\mathbb{N}$ is equivalent to $\mathsf{ATR}_0$.

In what follows we assume some basic knowledge of subsystems of second-order arithmetic, e.g., as presented in Simpson \cite{Simpson}. Below, the notation $\Delta^0_1$ refers to formulas which possibly contain second-order parameters.

\section{Binary Choice Games in Second-Order Arithmetic}\label{SectionDefinitions}
\begin{definition}[Binary choice tree]
 The set $T$ of (codes for) finite sequences of natural numbers $(\tau \in \N^{< \N})$ is a binary choice tree if it satisfies:\begin{enumerate}
    \item $\forall \tau \ \tau \in T \rightarrow (\forall n < |\tau|)\ \tau[n] \in T$ and;
    \item $\forall \tau \ \tau \in T \rightarrow \exists^{\leq 2} n \ \tau^\smallfrown n \in T$.
  \end{enumerate}
\end{definition}

\begin{definition}
  Let $T$ be a binary choice game tree. 
  \begin{enumerate}
    \item We define a \textbf{regular strategy} $\sigma$ (respectively, $\tau$) for player \I\ (respectively, \II) -- written $\sigma \in \SI$ (respectively $\tau \in \SII$) -- as any function $\sigma: \N^{\mathrm{even}} \to \N$ (respectively $\tau: \N^{\mathrm{odd}} \to \N$). We define $\sigma \otimes \tau$ as the unique sequence of natural numbers $\langle n_{i} \rangle_{i \in \mathbb{N}}$ produced by the alternation of these two strategies across the even and odd indices. 
That is, \begin{align*}
      n_{0} = \sigma(\emptyset) \land n_{i+1} = \left\{ \begin{array}{ll}
        \sigma(\langle n_0, \dots, n_i \rangle) &\text{if $i$ is even}; \\
        \tau(\langle n_0, \dots, n_i \rangle) &\text{if $i$ is odd}.
      \end{array} \right.
    \end{align*}
    \item We define a \textbf{restricted strategy} $\sigma$ for player \I\ in the binary choice game tree $T$ -- written $\sigma \in \SI(T)$ -- as a subtree of $T$ 
    that satisfies the properties: \begin{align*}
      &\forall t \in \N^{\mathrm{even}} \ (t \in \sigma) \rightarrow (\exists! n \ (t^{\smallfrown}n \in \sigma)) \quad \text{and} \\
      &\forall t \in \N^{\mathrm{odd}} \ (t \in \sigma) \rightarrow [\forall n \ (t^{\smallfrown}n \in T) \rightarrow (t^{\smallfrown}n \in \sigma)].
    \end{align*}
    A restricted strategy $\tau$ for player \II\ -- written $\tau \in \SII(T)$ -- is defined similarly, swapping the role of odd and even. We define $\sigma \otimes \tau$ as the intersection $\sigma \cap \tau$.
  \end{enumerate}
\end{definition}

\begin{remark}
For any binary choice game tree $T$ and any node $t \in T$, $\rca$ proves that $Succ(t) = \{ s : \exists n \ s = t^{\smallfrown}n \land s \in T \}$ exists by bounded $\Sigma^0_1$ comprehension (see Simpson \cite{Simpson}*{theorem II.3.9}).
\end{remark}

The subtlety of games on binary choice trees relies on the fact that the exact structure of the game tree is inaccessible in $\rca$. Even if $T$ is recursive, the set of nodes in $t$ that have exactly one successor might not be. 
In contrast, for binary games with moves in $\{0,1\}$, and unrestricted games on natural numbers, the game tree can be assumed to be the full tree of their respective space. For these games, the winning condition is the only interesting structure to take into account.

  Coming back to our binary choice game trees: we want to penalize the first player playing outside a tree $T$.
  To this aim, let us define $\mu_I(x)$, ranging over $\N^{\N}$ by the formula \begin{align*}
    &\exists n \ x[2n+1] \not\in T \land x[2n] \in T,
  \end{align*} 
  that is, $\mu_I$ asserts that player $\I$ played outside $T$ first.
  The formula $\mu_{II}(x)$ is defined similarly, by \begin{align*}
    &\exists n \ x[2n+2] \not\in T \land x[2n+1] \in T.
  \end{align*} 
Taking $\phi(x) \in \Gamma$, we can then define: 
\begin{definition}
\textbf{$\Gamma$-Determinacy for binary choice game trees} asserts that whenever $T$ is a binary choice game tree and $\phi \in \Gamma$, then the game $G(T,\phi)$ is determined, i.e.,
 \begin{align*}
    \exists \sigma \in \SI \forall \tau \in \SII \psi( \sigma \otimes \tau) \lor \exists \tau \in \SII \forall \sigma \in \SI \lnot \psi( \sigma \otimes \tau),
  \end{align*} 
  where $\psi(x) \equiv \lnot\mu_{\I}(x) \land (\phi(x) \lor \mu_{\II}(x))$. \label{binary_det}
\end{definition}

\begin{remark}\label{RemarkDifferenceGamma}
If $\phi \in \Gamma$,
$G(T, \psi)$ has a $\Pi^0_1\wedge (\Gamma \vee \Sigma^0_1)$ payoff set, but this payoff set is of a very particular kind, as the $\Pi^0_1$ set and the $\Sigma^0_1$ are both derived from the same tree $T$. For the classes $\Gamma$ we consider below (namely, $\Gamma = \Delta^0_1$ or $\Gamma = (\Sigma^0_1)_n$), this payoff set will generally be of complexity greater than $\Gamma$. By \cite[Corollary 21]{AgMSD}, these are essentially the only classes for which the complexity differs.
\end{remark}
 Let us state an equivalent definition, for which the geometry of the tree does not impact the complexity of the winning condition.

\begin{definition}\label{binary_det2}
  Let $\phi(x)$ be a formula of some class $\Gamma$ with a distinguished free set variable and $T$ be a binary choice game tree.
  We say that \textbf{the game $G(T,\phi)$ is determined in restricted strategies} if 
  \begin{align*}
      \exists \sigma \in \SI(T) \forall \tau \in \SII(\sigma) \phi(\sigma \otimes \tau) \lor \exists \tau \in \SII(T) \forall \sigma \in \SI(\tau) \lnot \phi( \sigma \otimes \tau).
  \end{align*} 
\end{definition}

We observe that in some cases, $\I$ can have a winning strategy that prohibits $\II$ from having any strategy, that is, that forces $\II$ to play out of the tree, vacuously satisfying the winning condition this way.

\begin{remark}
  Definitions~\ref{binary_det} and~\ref{binary_det2} are equivalent in $\rca$. That is for any formula $\phi$, and binary choice tree $T$, the game $G(T, \phi)$ is determined to the sense of~\ref{binary_det} iff it is determined to the sense of~\ref{binary_det2}, provably in $\rca$.
\end{remark}

We introduce now a notion of determinacy for wellfounded binary branching trees.

\begin{definition}
  Let $T$ be a wellfounded tree and $u$ be a finite sequence of natural numbers. We define \begin{align*}
    &W_{\I}(u) = \exists(n < |u|) \ u[n] \not\in T \land (\forall k < n) u[k] \in T \land n \text{ is odd};
    \\ &W_{\II}(u) = \exists(n < |u|) \ u[n] \not\in T \land (\forall k < n) u[k] \in T \land n \text{ is even}.
  \end{align*}
  We then say that the game on $T$ is determined iff \begin{align*}
    \exists \sigma \in \SI \forall \tau \in \SII \ W_{\I}(\sigma \otimes \tau \cap T) \lor \exists \sigma \in \SII \forall \tau \in \SI \ W_{\II}(\sigma \otimes \tau \cap T).
  \end{align*}
  
\end{definition}

Notice that $W_{\I}(\sigma \otimes \tau \cap T)$ can be expressed in a $\Delta^0_1$ fashion since $T$ is assumed to be wellfounded. Indeed, it can be expressed equivalently in the following ways \begin{align*}
  \exists n \ W_{\I}((\sigma \otimes \tau)[n]), \text{ and }
\end{align*}
\begin{align*}
\forall n \ (\sigma \otimes \tau)[n] \not\in T \rightarrow W_{\I}((\sigma \otimes \tau)[n]).
\end{align*}

Therefore, determinacy for wellfounded binary trees is a suitable a priori weaker notion than $(\Sigma^0_1)_k-\Det$ for binary choice game trees, for $k \geq 2$.

\section{Proof of Theorem \ref{TheoremMain}}

\begin{theorem}
Fix $k \geq 2$. The following are equivalent over $\rca$:
\begin{enumerate}
\item Determinacy for \emph{wellfounded} binary choice game trees.
\item $\Delta^0_1$-$\Det$ for binary choice game trees.
\item $(\Sigma^0_1)_k$-$\Det$ for binary choice game trees.
\item $\aca$.
\end{enumerate}
\end{theorem}

 \begin{proof}
First, we fix $k \geq 1$, assume $\aca$, and prove that $(\Sigma^0_1)_k$-$\Det$ for binary choice game trees holds. Given a binary choice game tree $T$, we define an embedding $\rho: T \to 2^\mathbb{N}$ by: $\rho(s^\smallfrown n) = \rho(s)^\smallfrown 0$ if $s^\smallfrown n$ is the leftmost successor of $s$ in $T$ and $\rho(s^\smallfrown n) = \rho(s)^\smallfrown 1$ if $s^\smallfrown n$ is a successor of $s$ in $T$ but not the leftmost successor. Observe that $\rho$ is recursive in $T$. By $\aca$, the tree $T' = \text{range}(\rho)$ exists as a set and is a subtree of $2^\mathbb{N}$. By Remark \ref{RemarkDifferenceGamma}, every $(\Sigma^0_1)_k$ game on the tree $T'$ belongs to the class $(\Sigma^0_1)_{k+2}$ (as a subset of the full binary tree) and thus is determined, according to the theorem of Nemoto, MedSalem, and Tanaka \cite{NeMSTa07}. Using $\rho$ and $\aca$, a winning strategy for this game can be pulled back into a winning strategy for the game on $T$.

Trivially, $3.$ implies $2.$, which implies $1$.
To complete the proof of the theorem, we show that $\aca$ follows from determinacy for wellfounded binary choice game trees. We appeal to the following well known result (see Simpson \cite{Simpson}*{theorem III.7.2}): $\aca$ is equivalent to weak K\"onig's lemma for binary choice trees. 
 
Thus, we suppose for the sake of a contradiction that there is an infinite binary choice tree $T$ that does not have any branch. 
We will suppose without loss of generality that $T$ does not contain any zeroes in its constitutive sequences. 
Otherwise, we just replace $T$ by the tree $T'$ defined by  $$\langle n_0, n_1, \cdots n_i \rangle \in T' \leftrightarrow \bigwedge_{j = 0}^i (1 \leq n_i) \land \langle n_0-1, n_1-1, \cdots n_i-1 \rangle \in T,$$ 
which is also binary branching, also violates weak König's lemma, and exists by $\Delta^0_1$-Comprehension.

Let us first explain the game we consider, then we will show it has a wellfounded binary choice game tree, and finally we show how to obtain the conclusion from it. \textit{The binary choice game tree will be different from $T$.}
    
    The game is divided in four phases. During phase $1$, player \I\ can construct any sequence $t \in T$. Then, during phase $2$, player \II\ answers with $u_0\in\mathbb{N}$ with $t^{\smallfrown} u_0 \in T$. In phase $3$, player \I\ produces a sequence $v$, intending to extend $t$, with $v(0) \neq u_0$. Finally, during phase $4$, player \II\ produces a sequence $u'$, intending to extend $t^{\smallfrown}u_0$. Player II's goal is for her sequence to be longer than $\I$'s.
    
\begin{center}
\begin{tabular}{lllll}
I: & $t$ & & $v$ & \\
II: & & $u_0$ & & $u'$
\end{tabular}
\end{center}

We now describe how to encode this game with a binary choice game tree. We explain how player I plays the sequence $t$; the sequences $v$ and $u'$ are played similarly.
Suppose a sequence $s \in T$ has been constructed so far by player $\I$, and they would like to continue extending $s$. Then
    \begin{enumerate} 
      \item\label{Playn} First, player $\I$ plays $n$, the leftmost successor of $s$ in $T$.
      \item As a second move, $\I$ plays either $0$ to confirm that $n$ is the next element to be added to the sequence (recall that by choice of $T$, we have $s^{\smallfrown}0 \not\in T$), or else $\I$ plays the rightmost successor $m$ of $s$ in $T$. 
      In this case, $m$ is added as the next element of the sequence.
    \end{enumerate} 
The new sequence is now $s^{\smallfrown}n$ or $s^\smallfrown m$. 
After this, $\I$ either indicates that they are not done constructing their sequence by playing $0$ and repeat the above process, or plays $1$ to indicate that the next phase of the game is to start. Throughout this process, player $\II$ must play $0$ each turn, as her moves are irrelevant.

Observe that throughout the construction of the sequence $t$, each player has at most two legal moves each turn. Moreover, whether a play by $\I$ is legal requires checking finitely many instances of membership in $T$, so the binary choice game tree defined thus far is recursive.

Here is an illustration.
\\\phantom{a}\\
\begin{tabular}{lllllllllllll}
    \I:\ & $0$ & & $0?n_0$ & &$0?m_0$ & & $0$ & & $0?n_1$ & & $0?m_1 \ \cdots$ &\\
    &&&&&&&&&&&&\\
          
    \II:\ & & $*$ & & $*$ & & $*$ & & $*$ & & $*$ & & $\cdots$

\end{tabular}
\\\phantom{a}
\\\phantom{aaaaaaaaaaaaaaasss}
\begin{tabular}{lllllllllll}
  & $\cdots$ & & $0$ & & $0?n_{|t|-1}$ & &$0?m_{|t|-1}$ & & $1:$signal &\\
  &&&&&&&&&&\\
  & & $\cdots$ & & $*$ & & $*$ & & $*$ & & $[u_0]$
\end{tabular}
\\\phantom{a}\\
For example, here $t$ is given by $\langle p_0, p_1, \cdots p_{|t|-1} \rangle$. The symbol $*$ indicates that player $\II$'s moves are not relevant, although we formally require that she play zeroes. The play $u_0$ and the sequences $v$ and $u'$ are encoded similarly using $T$, \textit{except} that when playing $u_0$, we allow $n = 0$ in \eqref{Playn}, in which case the next play by $\II$ must also be $0$ (confirming the choice $n = 0$). This is done for a technical reason. Note that in this case $t^\frown u$ will not belong to $T$.
We observe that player $\I$'s legal moves require that $t^\frown v \in T$, and player $\II$'s legal moves (given by the game tree) require that $t^\smallfrown u \in T$, where $u = \langle u_0 \rangle^{\smallfrown}u'$, except possibly if $u_0 = 0$.

The winning condition for $\II$ is the set consisting of all plays satisfying:
\begin{equation}\label{PlayerIIWC}
v(0) = u_0 \lor v = \emptyset \lor \big( |v| \leq |u| \wedge t^\smallfrown u \in T\big).
\end{equation}
Since $T$ has no branches by hypothesis, there are no infinite plays of the game, so it is indeed a wellfounded binary choice game tree.

According to the definition of games on binary choice game trees, we have the following: if one player plays a move outside of the game tree, then the first one to do so loses. Otherwise, if both players play within the game tree, player $\II$ wins if and only if \eqref{PlayerIIWC} holds. By construction of the game tree, we have the following:
\begin{enumerate}
\item\label{RuleSumm1} If $t^\smallfrown v \not \in T$, then player $\II$ wins;
\item\label{RuleSumm2} else, if $t^\smallfrown v \in T$ but $v = \emptyset$ or $v(0) = u_0$, then player $\II$ wins; 
\item\label{RuleSumm3} else, if $t^\smallfrown u \not\in T$, then player $\I$ wins;
\item\label{RuleSumm4} else player $\II$ wins if and only if $|v| \leq |u|$.
\end{enumerate}

Let us show that \I\ cannot have a winning strategy. 
Suppose otherwise and let $\sigma$ be a winning strategy. Let us consider $t$, the sequence $\sigma$ constructs in the first phase of the game. We must have that $t \in T$, as otherwise \II\ would win because of \eqref{RuleSumm1}. We also must have that $t$ has a successor $n$ in $T$, otherwise \II\  would play $n = 0$ and win because of \eqref{RuleSumm2}. 
Since $\sigma$ is winning, $t$ must be splitting, i.e., have two immediate successors $n$ and $m$, for if it only had one, it could be chosen as $u_0$ and player $\II$ would win by \eqref{RuleSumm2}.  

Let $T_n = \{ s : t^{\smallfrown}n^{\smallfrown}s \in T \}$ and $T_m$ be defined similarly. There are two cases:
\begin{itemize}
  \item Either $T_n$ or $T_m$ is bounded. Then player \II\ can win the game by choosing the index having the longest extensions (or with unbounded extensions) in $T$ as $u_0$, and then choosing $u$ so that $|v| \leq |u|$.
  \item Both $T_n$ and $T_m$ are unbounded. Then, whatever \II\ chooses, she can produce a longer sequence in $T$ than $\I$, again ensuring $|u| \leq |v|$.
\end{itemize}
In both cases, we obtain a contradiction to the fact that $\sigma$ was a winning strategy. 

By determinacy for wellfounded binary choice game trees, \II\ has a winning strategy $\tau$. 
We then define $f: \N \to \N$ as \begin{align*}
  f(n) = u_0(\tau(f[n])),
\end{align*}
i.e., the answer $u_0$ of $\tau$ to \I\ playing $t = f[n]$, where $f[n]$ is the string obtained by concatenating $f(k)$ for $k < n$.

Recall that $\Sigma^0_1$-induction is one of the axioms of $\rca$. We use it to prove that 
\begin{align*}
\forall n\, \theta(n), \qquad \text{where } \theta(n) = [f[n] \in T \wedge \exists k\, \big( f[n]^{\smallfrown}k \in T \big)].
\end{align*}
Note that $\theta(n)$ is indeed a $\Sigma^0_1$ formula, with $T$ and $\tau$ as parameters.
First $f[0] = u_0(\tau(\emptyset)) \in T$ since $T$ is non-empty and otherwise \I\ could defeat this play by playing $v$ in the tree. Moreover, $f[0]$ has a successor in the tree. Otherwise, since $T$ is infinite (in particular $|T| > 3$), \I\ could have played an extension of $\emptyset$ incompatible with $f[0]$ and  of length $\geq 2$, defeating a play consistent with $\tau$, which is impossible.

We prove by induction that $f[n]$ belongs to $T$ and has a successor in $T$. We first consider the case $n = 1$ explicitly for the sake of exposition.
Since we proved $f[0]$ has a successor, $f[1] = f(0)^\smallfrown u_0(\tau(f[0]))$ is in $T$, since, again, otherwise \I\ could defeat $u_0(\tau(f[0]))$ as above by playing $v$ in $T$.

Suppose towards a contradiction that
$$\lnot \diamond^1(f[1]) := \forall n \ f[1]^{\smallfrown}n \not\in T$$
holds; thus $|T_{f[1]}| = 1$. Then because $f[1] = f(0)^\frown u_0(\tau(f[0]))$ and $\tau$ is winning,  we have
$$\lnot\diamond^2(f[0]) := \forall n,m \ f[0]^{\smallfrown}\langle n,m \rangle \not\in T,$$ 
for otherwise if $v = \langle n,m \rangle$ witnesses $\diamond^2(f[0])$,
then it follows from $\lnot\diamond^1(f[0])$ that $\I$ could have defeated $\tau$ by playing $v$ against $u = \langle u_0 \rangle = \langle f(1)\rangle $; thus, $|T_{f[0]}| \leq 3$. 
In a similar way we obtain
$$\lnot\diamond^3(\emptyset) := \forall n,m,k \ \langle n,m,k \rangle \not\in T.$$ 
This is because otherwise if $v = \langle n,m, k\rangle$ witnesses $\diamond^3(\emptyset)$, then it follows from $\lnot\diamond^2(f[0] )$ that $\I$ could have defeated $\tau$ by playing $v$ after $\tau$ plays $u_0 = f[0]$; hence $|T| = |T_\emptyset| \leq 7$. 
But this is a contradiction to the assumption that $T$ is infinite, hence the conclusion follows. See Figure \ref{FigureTreeCase1} for a picture.

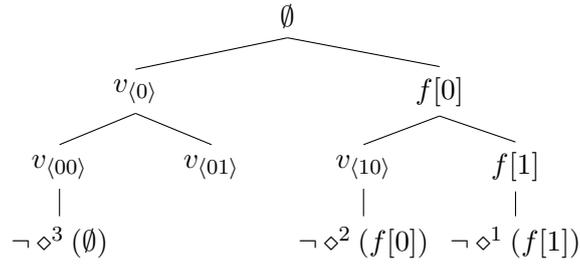
\begin{figure}
  \
  \center
  \begin{tikzpicture}[
  level 1/.style = { level distance   = 10mm,
                     sibling distance = 40mm },
  level 2/.style = { level distance   = 10mm,
                     sibling distance = 20mm },
  level 3/.style = { level distance   = 10mm,
                     sibling distance = 10mm } ]
  \node {$\emptyset$}
    child { node {$v_{\langle 0 \rangle}$}
      child { node {$v_{\langle 0 0 \rangle}$} 
        child { node {$\lnot\diamond^3(\emptyset)$} }
      }
      child { node {$v_{\langle 0 1 \rangle}$} }      
    }
    child { node {$f[0]$} 
      child { node {$v_{\langle 10 \rangle}$} 
        child { node {$\lnot\diamond^2(f[0])$} }
      }
      child { node {$f[1]$}
        child { node {$\lnot\diamond^1(f[1])$} }
      }
    }
  ;
 \end{tikzpicture}
 \caption{Proof that $f[1]$ has a successor in $T$.}\label{FigureTreeCase1}
\end{figure}

The general case follows a similar line of reasoning. Suppose that $f[n] \in T$ and has a successor in $T$. Hence, $f[n+1] \in T$, otherwise $\tau$ would lose after $\I$ plays this successor as $v(0)$. 
Now we need to prove that $f[n+1]$ has a successor in $T$. 
We claim that:
\begin{equation}\label{eqSubsidiaryInduction}
|T_{f[n+1-j]}| \leq 2^{j+1}-1.
\end{equation}
Observe that \eqref{eqSubsidiaryInduction} is equivalent to the assertion ``all sequences of length $2^{j+1}$ contain an element not in $T_{f[n+1-j]}$'' and thus is a $\Pi^0_1$ assertion about $j$ with $T$ and $\tau$ as parameters.
We prove \eqref{eqSubsidiaryInduction} by a subsidiary $\Pi^0_1$-induction on $j \leq n+1$.
The base case asserts that $|T_{f[n+1]}| \leq 1$ and is precisely the assumption that $f[n+1]$ has no successor in $T$. The inductive step is established by a similar argument as above, through a process extending that of Figure \ref{FigureTreeCase1}, using the fact that $\tau$ is a winning strategy.
Thus, by $\Pi^0_1$-induction, we derive the claim \eqref{eqSubsidiaryInduction} which in particular implies that $|T| \leq 2^{n+2}-1$,
contradicting the assumption that $T$ was infinite.
Hence, we conclude that $f[n+1]$ indeed has a successor in $T$.
By the main $\Sigma^0_1$-induction, we conclude $\forall n \ f[n] \in T$, showing that $T$ has a branch, which is a contradiction. This completes the proof.
\end{proof}

\section*{Acknowledgements}
This work was partially supported by FWF grants P36837, STA-139. The authors would also like to acknowledge the support of the Erwin Schr\"odinger Institute as part of the thematic program ``Reverse Mathematics'' in 2025.

\bibliographystyle{abbrv}

\end{document}